\documentclass[11pt]{article}
\usepackage[utf8]{inputenc}
\usepackage{amsmath,amsthm}
\usepackage[mathscr]{euscript}
\usepackage{amsfonts}
\usepackage{amssymb}
\usepackage{lmodern}
\usepackage{fourier}
\usepackage{graphicx}

\newtheorem*{theorem}{Theorem}
\newtheorem{lemma}{Lemma}

\newtheorem{definition}{Definition}

\newcommand{\fla}{f_{\lambda,a}}

\begin{document}
\title{\textbf{ Boundedness of Fatou components of the family\\
$ f_{ \lambda }(z)= \lambda \sin(z)+a $}}
\author{FERNANDO REN\'{E} MART\'{I}NEZ ORTIZ\thanks{This work was supported by CONACYT grant 177246}\\
     \emph{Universidad Aut\'{o}noma de la Ciudad de Mexico (UACM)}\\
     \emph{Calle Prolongaci\'{o}n San Isidro No. 151 Col. San Lorenzo Tezonco,}\\
     \emph{Alcald\'{i}a Iztapalapa, Ciudad de M\'{e}xico, C.P. 09790.}
     \and
     GUILLERMO SIENRA LOERA\thanks{ and by PAPIIT IN106719}\\
     \emph{Facultad de Ciencias, UNAM}\\
     \emph{Av. Universidad 3000, C.U., Ciudad de M\'{e}xico, C.P. 04510.}
     }
\date{Oct 8, 2019}
\maketitle

\begin{abstract}
    Several partial results have been obtained in order to prove that the components
    of the Fatou set of the family $ f_{ \lambda }(z)= \lambda \sin(z) $ are bounded for 
    $ \left| \lambda \right| \geq 1 $. In particular, Dom\'{i}nguez and Sienra
    have given a proof for values of $ \lambda $ on the unit circle of parabolic type;
    Gaofei Zhang has proved that, in the case $ \lambda = e^{2 \pi i \theta} $ with 
    $ \theta $ an irrational number of bounded type, $ f_{ \lambda }(z)= \lambda \sin(z) $
    has an invariant bounded Siegel disk and all the others components (bounded) are its 
    preimages. Furthermore, Dom\'{i}nguez and Fagella have proved the boundedness of all the 
    Fatou components for the sine family in the case 
    $ \left| Re\left( \lambda \right)  \right| > \frac{\pi}{2} $ based in the
    construction of a dynamic hair via exponential tracts. Finally, for hyperbolic
    functions in a more general setting, the boundedness has been proved by Bergweiler,
    Fagella and Rempe--Gillen \cite{BFR2015}.\\
    
    In this paper, we give new proofs related to the unboundedness and complete invariance
    of the Fatou component that contains $ 0 $ when 
    $ 0<\left| \lambda \right| < 1 $ and also of the boundedness of all the Fatou 
    components when $ \lambda = e^{2 \pi i \frac{p}{q}} $, where 
    $  p,q \in \mathbb{Z}, q \neq 0 $.\\
    
    In our first main theorem, we prove that the components of the Fatou set are 
    bounded in the case $ \left| \lambda \right| > 1 $ with the additional
    hypothesis that the post--singular set of $ f_{ \lambda } $ is bounded using
    the existence of hairs guaranteed by the work of Benini and Rempe--Gillen.\\
    
    Finally, in the last section, we extend the last result ($ \vert \lambda \vert>1 $ 
    and post--critical set bounded)
    to the case of the extended sine family $ \fla=\lambda \sin(z)+a $. We prove that
    if does not exist a Fatou domain that contains all the critical points then all
    the components of the Fatou set of $ \fla $ are bounded. Also, we prove that, if
    there exists an unbounded Fatou domain forward invariant, then $ \fla $ must be
    of disjoint type.
    
\end{abstract}

\section{Introduction} \label{S:intro}

We consider the iteration of transcendental entire functions in the family
$ f_{ \lambda }(z)= \lambda \sin(z) $
, $ \lambda \in \mathbb{C}^* = \mathbb{C} \setminus \lbrace 0 \rbrace $.
This family is contained in the Eremenko--Lyubich class $ \mathscr{B} $,
class that consists of those transcendental entire functions for which the
set of singularities, $ S(f) $ (critical values together with asymptotic values),
is bounded. In fact, in the sine family, the set of singularities is finite, so
this family is contained in the Speiser class, $ \mathscr{S} $.\\

The sequence of iterates produces the well known dichotomy between the Fatou set of
$ f $, denoted by $ F(f) $, defined as the set of those points $ z \in \mathbb{C} $ such
that the sequence $ \left( f^{n}  \right)_{n \in \mathbb{N}}  $ forms a normal family
in some neighbourhood $ U $ of $ z $, and its complement, the Julia set of $ f $, denoted
by $ J(f) $. Good introductions into the subject are Beardon \cite{B1991}, Milnor 
\cite{M1999} and Hua \cite{HY1998}. An excellent and very insightful survey about the Eremenko--Lyubich class is Sixsmith \cite{S2018}.\\

Well known properties of this sets that occur in very general settings and in
particular for transcendental entire functions in the class $ \mathscr{B} $ are
the following: 

\begin{enumerate}
	\item $ F(f) $ is open and $ J(f) $ is closed, $ J(f) \neq \varnothing $.
	\item Both sets are completely invariant.
	\item For every positive integer $ k, F(f^{k})=F(f) $ and $ J(f^{k})=J(f) $.
	\item $ J(f) $ is perfect.
	\item If $ J(f) $ has nonempty interior, then $ F(f)= \varnothing $.
	\item $ J(f) $ is the closure of the set of repelling periodic points of $ f $.
\end{enumerate}

Since wandering domains are impossible for $ f \in \mathscr{S} $, the only
possible periodic Fatou components that can appear in the sine family are
attracting domains, parabolic domains and Siegel disks. \cite{EL1992}\\

Eremenko and Lyubich in \cite{EL1992} showed that, if $ f \in \mathscr{B} $,
then the escaping set, $ I(f) $, is contained in the Julia set of $ f $, where
\[
	I(f)=\lbrace z \in \mathbb{C} \vert f^{n}(z) \to \infty \rbrace.
\]

The following properties of the set $ I(f) $ were proved by Eremenko \cite{E1989} for a
transcendental entire function $ f $.

\begin{enumerate}
	\item $ I(f) \neq \varnothing $.
	\item $ J(f)= \partial I(f) $.
	\item $ I(f)\cap J(f) \neq \varnothing $.
	\item $ \overline{I(f)} $ has no bounded components.
\end{enumerate}

In that paper, Eremenko made the conjecture that, for an arbitrary transcendental entire
function, \emph{''It is plausible that the set $ I(f) $ always has the following 
property: every point} $ z \in I(f) $ \emph{can be joined with} $ \infty $
\emph{by a curve in} $ I(f) $ \emph{''}.\\

However, even in the small class $ \mathscr{B} $, Rottenfusser, R\"{u}ckert,
Rempe and Schleicher \cite{RRRS2011} have shown that there exists an hyperbolic entire function $ f $ such that every path connected component of $ J(f) $ is bounded, so
Eremenko's conjecture is false.\\

Despite of this result in the negative, they have proven that Eremenko's conjecture
is true if $ f $ is a finite composition of functions in the class $ \mathscr{B} $
of finite order of growth, i.e., $ f=g_{1}\circ \cdots \circ g_{n} $ such that for all
$ i \in \lbrace 1,\ldots,n \rbrace $, $ g_{i} \in \mathscr{B} $ and
\[
	\log{\log{\vert g_{i}(z) \vert}} = O( \log{\vert z \vert}) \qquad \mbox{as} \qquad \vert z \vert \to \infty. 
\] 

The above results are very important for us because, although by other methods
(developed by Devaney and Tangerman \cite{DT1986}), the existence of these curves dynamically defined and contained in $ J(f) $ has been used by Dom\'{i}nguez and Fagella
\cite{DF2008} in order to show that the Fatou components of 
$ f_{ \lambda }(z)= \lambda \sin(z) $ are bounded
in the case $ \left| Re\left( \lambda \right)  \right| > \frac{\pi}{2} $. In fact, the
proofs of our results in the parabolic and repelling cases are based in the
arguments given in \cite{DF2008} and in Rempe \cite{R2008}. \\

\begin{definition}
	An \textbf{invariant hair} of a transcendental entire function $ f $ is a continuous
	and injective maximal curve $ \gamma: \mathbb{R}^{+} \rightarrow I(f) $ such that 
	$ f(\gamma(t))=\gamma(t+1) $ for all $ t $ and 
	$ \lim_{t \rightarrow \infty}\vert f(t) \vert = \infty $. A \textbf{periodic hair} is
	a curve that is an invariant hair for some iterate $ f^{n} $ of $ f $. If the limit
	$ \lim_{t \rightarrow 0} \gamma(t)= z_{0} $ exists, we say that the curve $ \gamma $
	\textbf{lands} at $ z_{0} $.
\end{definition}

We recall that the post--singular set of $ f $ is the set defined as

\[
\mathcal{P}(f)= \overline{\bigcup_{c \in S(f)}{\lbrace f^{n}(c) \vert n \geq 0 \rbrace}}
\]

Benini and Rempe--Gillen \cite{BR2017} have proven that if $ f \in \mathscr{B} $ and $ f $ has finite order of growth (in fact, they state and prove their theorem in a more general setting) and $ \mathcal{P}(f) $ is bounded, then every periodic hair of $ f $ lands at a repelling or parabolic point and conversely, every repelling or parabolic point of $ f $ is the landing point of at least one and at most finitely many periodic hairs.\\

In this paper we give a new proof for each of the following theorems which were proved
by Dom\'{i}nguez and Sienra in \cite{DS2002}

\begin{theorem}[Attracting Case] \label{T:attracting}
	For the map $ f_{ \lambda }(z)= \lambda \sin(z) $ where $ 0<\vert \lambda \vert<1 $
	the Fatou set of $ f_{ \lambda } $ consists of a simply connected, completely
	invariant component $ U $. 

\end{theorem}

\begin{theorem}[Parabolic Case] \label{T:parabolic}
	Let $ f_{ \lambda }(z)= \lambda sin(z) $ with $ \lambda = e^{2 \pi i \frac{p}{q}} $, 
	$  p,q \in \mathbb{Z}, \quad q \neq 0 $ and $ (p,q)=1 $
	then all the components of $ F(f_{ \lambda }) $ are bounded. 

\end{theorem}

Finally, using the result of Benini and Rempe--Gillen mentioned above, we can prove that the
components of the Fatou set of $ f_{\lambda} $ are bounded for some values of 
$ \lambda $ with $ \vert \lambda \vert >1 $ as follows. 

\begin{theorem}[Repelling Case] \label{T:repelling}
	If $ \vert \lambda \vert>1 $ and $ \mathcal{P}(f_{ \lambda }) $ is bounded, then 
	all the components of $ F(f_{ \lambda }) $ are bounded. 

\end{theorem}

\begin{definition}
	A map $ f $ is of \textbf{disjoint type} if $ f $ is hyperbolic with connected
	Fatou set.
\end{definition}

In the last Section, we use the same ideas on a more general family, the extended sine
family $ \fla(z)= \lambda \sin (z)+a $. We prove the following two theorems: 

\begin{theorem}[Repelling Case in Extended Family]\label{T:extended1}
	Let $ \fla(z)=\lambda \sin (z) + a $ with $ \vert \lambda \vert  > 1 $ a
	post--singularly bounded function, then all the components of the
	Fatou set of $ \fla $ are bounded if and only if for each of them exists a critical
	point that is not in the component.
\end{theorem}

\begin{theorem}[Disjoint Type]
	Suppose that $ U $ is an unbounded component of the Fatou set of $ \fla $, where
	$ \vert \lambda \vert > 1 $ and $ \mathcal{P}(\fla) $ is bounded. If $ U $ is forward 
	invariant then $ \fla $ is of disjoint type.
\end{theorem}

\textbf{Acknowledgments}\\

The authors would like to thank to Patricia Dom\'{i}nguez Soto and Lasse Rempe--Gillen
by their many helpful comments about the manuscript: finding many mistakes, suggesting
improvements and pointing out to us references about previous work.

\section{Attracting Case} \label{S:one}

Consider the map $ f_{\lambda}(z)= \lambda \sin(z) $, we remark that 
$ \lambda \sin(-z)= - \lambda \sin(z) $, it implies that the orbits
of $ -z $ and $ z $ are symmetrical and thus the sets $ F(f_{\lambda}) $ and 
$ J(f_{\lambda}) $ have central symmetry with respect to the origin.\\

The Julia and Fatou sets of $ f_{\lambda} $ have another important central symmetry around
$ \frac{\pi}{2} $. That is immediate using the identities
\[
	\sin\left( \frac{\pi}{2} - z \right)= \cos(z) =\sin\left( \frac{\pi}{2} + z \right) 
\]  

So, if $ z_{1}= \frac{\pi}{2} + w $ and $ z_{2}= \frac{\pi}{2} - w $ then
$ f_{\lambda}(z_{1})=f_{\lambda}(z_{2}) $, consequently $ z_{1} \in F(f_{\lambda}) $
if and only if $ z_{2} \in F(f_{\lambda}) $, see Figure \ref{fig0}.\\

\begin{figure}[ptb]
	\begin{center}
		\includegraphics[width=0.8\linewidth]{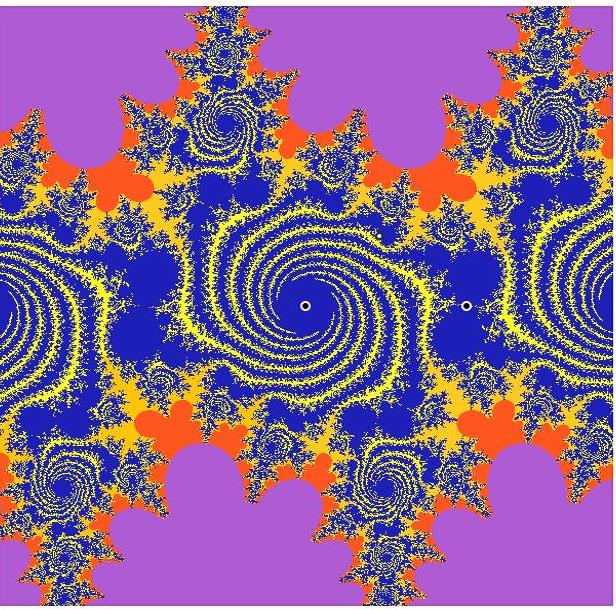}
	\end{center}
	\caption{Note the rotational symmetry around each of the marked points $ 0 $ and
	$ \frac{\pi}{2} $. This is the dynamical plane in the case 
	$ \lambda =0.7229 + 0.6981 i $}
	\label{fig0}
\end{figure}

Now, we can prove:\\

\begin{theorem}[Attracting Case] 
	For the map $ f_{ \lambda }(z)= \lambda \sin(z) $ where $ \vert \lambda \vert<1 $
	the Fatou set of $ f_{ \lambda } $ consists of a simply connected, completely
	invariant component $ U $. 

\end{theorem}

\begin{proof}
We know that $ \lambda $ and $ - \lambda $ are in the component of the Fatou set 
$ W_{0} $ that contains the origin since they are critical values, thus, there exists
a simple curve $ \gamma \subset W_{0} $, symmetrical with respect to the origin
connecting $ \lambda, \quad 0 $ and $ - \lambda $ in this order. Let 
$ \Gamma = f_{\lambda}^{-1}(\gamma) $, then $ \Gamma $ has the same central symmetry
that $ \gamma $, contains all the singular points, is unbounded and is contained in
$ W_{0} $. Furthermore, we can suppose that $ \Gamma $ is a simple curve.\\

Now, consider all the inverse images of $ \Gamma $ under $ f_{\lambda} $ and denote by
$ \Gamma_{k} $ that one than contains the point $ \pi k $. Clearly, all these curves
are contained in $ W_{0} $. Denote as $ \Gamma_{k}^{+} $ the part of the curve
$ \Gamma_{k} $ that is above of $ \Gamma $, see Figure \ref{fig1}.\\

\begin{figure}
	\begin{center}
		\includegraphics[width=0.8\linewidth]{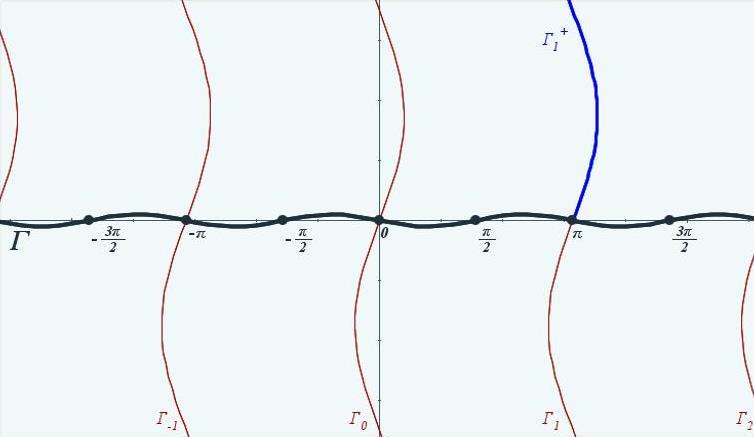}
	\end{center}
	\caption{ The curves $ \Gamma $, $ \Gamma_{k} $ and $  \Gamma_{k}^{+} $}
	\label{fig1}
\end{figure}

Suppose that $ W_{0} $ is not completely invariant, then exists a Fatou component
$ W \neq W_{0} $ such that $ f_{\lambda}(W)=W_{0} $. We can suppose w.l.o.g. that
$ W $ is contained in the region between $ \Gamma_{k_{0}}^{+} $, 
$ \Gamma_{k_{0}+2}^{+} $ and $ \Gamma $ for some $ k_{0} \in \mathbb{Z} $. By the 
periodicity of $ f_{\lambda} $ there exists a Fatou component 
$ W_{1}, \quad W_{1} \neq W_{0} $ ($ W_{1} $ is a translate of 
$ W $) such that $ f_{\lambda}(W_{1})=W_{0} $ and
$ W_{1} $ is contained in the region $ A $ determined by $ \Gamma_{0}^{+} $, 
$ \Gamma_{2}^{+} $ and $ \Gamma $, see Figure \ref{fig2}\\

\begin{figure}
	\begin{center}
		\includegraphics[width=0.8\linewidth]{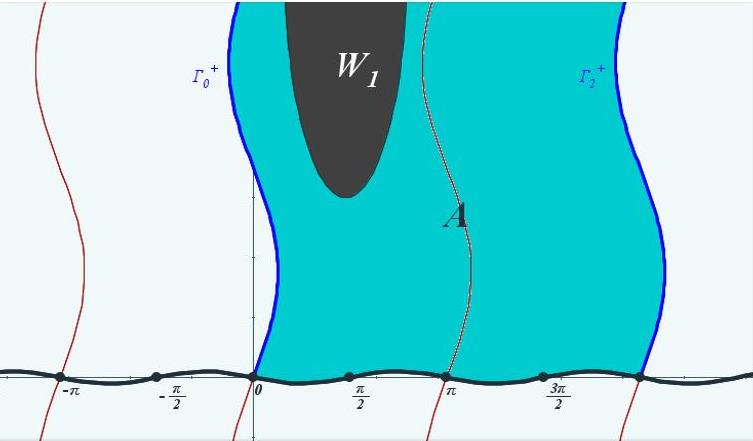}
	\end{center}
	\caption{$ W_{1} \subset A $}
	\label{fig2}
\end{figure}

If we consider $ f_{\lambda} $ restricted to the region $ A $, $ f_{\lambda} $ is
$ 1-1 $ and onto its image 
$ B= f_{\lambda}(A)= \mathbb{C}\setminus\left( \gamma\cup\Gamma^{*} \right)   $,
where $ \Gamma^{*}=f_{\lambda}(\Gamma_{0}^{+})=f_{\lambda}(\Gamma_{2}^{+}) $.\\
 
Now, let $ \delta >0 $ such that 
$ f_{\lambda}(B_{\delta}(0))\subseteq B_{\delta}(0) \subseteq W_{0} $ 
($ \delta $ exists because $ 0 $ is an attractive fixed point) and let
$ z_{0} \in B_{\delta}(0) \cap A $, then it follows that 
$ z_{1} = f_{\lambda}(z_{0})\in B \cap W_{0} $, but, because 
$ f_{\lambda}(W_{1}) = W_{0}  $, exists $ z_{2} \in W_{1} \subseteq A $ such
that $ f_{\lambda}(z_{2}) = z_{1} $ contradicting the fact that $ f_{\lambda} $
is $ 1-1 $ on $ A $. So $ W_{0} $ is completely invariant.\\

Furthermore, because $ W_{0} $ is unbounded, it is simply connected by Theorem 1 in Baker \cite{B1975}.
\end{proof}

\section{Parabolic Case} \label{S:two}

For the parabolic case we need some previous results.\\

\begin{definition}
	Let $ f $ be an entire function. $ f $ is \textbf{criniferous} if for every
	$ z \in I(f) $ and for all sufficiently large $ n $ there is an arc 
	$ \gamma_{n} $ connecting $ f^{n}(z) $ to $ \infty $, in such a way that
	$ \gamma_{n+1}=f(\gamma_{n}) $ and 
	$ \min_{z \in \gamma_{n}}\vert z \vert \rightarrow \infty $.
\end{definition}

It has been shown in \cite{RRRS2011} that if $ f \in \mathscr{B} $ and is a finite
composition of functions of finite order of growth, then $ f $ is criniferous.\\

In our case, we have that $ \sin(z) $ has finite order of growth since
\begin{align*}
	\vert \sin(z) \vert &\leq \vert \cosh \vert y \vert \vert = \frac{e^{\vert y \vert}+e^{-\vert y \vert }}{2} \\
	&\leq \frac{e^{\vert y \vert}+e^{\vert y \vert }}{2}= e^{\vert y \vert} \leq e^{\vert z \vert}	
\end{align*}

and, because all the polynomials have finite order of growth, $ f_{\lambda} $ is 
criniferous for all $ \lambda \in \mathbb{C}^{*} $.\\

We will use the following theorem due to Benini and Rempe--Gillen \cite{BR2017}
in order to get periodic hairs in the parabolic and repelling cases 
(with post--singular set bounded).

\begin{theorem}[Benini, Rempe--Gillen] \label{T:Benini}
	Let $ f $ be a transcendental entire function such that $ \mathcal{P}(f) $ is
	bounded. Then every periodic ray of $ f $ lands at a repelling or parabolic
	periodic point. If, in addition, $ f $ is criniferous, then conversely every
	repelling or parabolic periodic point of $ f $ is the landing point of at
	least one and at most finitely many periodic hairs. 
\end{theorem}

Previous work behind this important theorem was made, for the first part of the Theorem
and with this generality, by Rempe \cite{R2008}. The last part of this Theorem was
proved by Mihaljevi\'{c}--Brandt \cite{M2010} under the assumption that $ f $ is
geometrically finite.

\begin{lemma} \label{L:inH}
	If $ \gamma $ is a periodic hair of $ f_{\lambda} $ landing at $ 0 $, then exists 
	$ t_{0} \in \mathbb{R}^{+} $ such that either $ \gamma (t) $ is contained in the
	upper half--plane for all $ t > t_{0} $ or $ \gamma (t) $ is contained in the lower
	half--plane for all $ t > t_{0} $. 
\end{lemma}

\begin{proof}
	Suppose that the conclusion is false, then exists a sequence 
	$ ( t_{n} )_{n \in \mathbb{N}} $, $ t_{n} > t_{0} $ such that 
	$ \gamma(t_{n}) $ is on the real axis and $ t_{n}\rightarrow \infty $ when
	$ n\rightarrow \infty $. So, for all $ n \in \mathbb{N} $,
	$ f_{\lambda}(\gamma(t_{n})) $ is contained in the segment joining 
	$ - \lambda $ and $ \lambda $. This contradicts the fact that 
	$ f_{\lambda} \circ \gamma $ is a periodic hair that converges uniformly to
	$ \infty $.
\end{proof}

\begin{lemma} \label{L:realbounded}
	If $ \gamma $ is a periodic hair of $ f_{\lambda} $ that lands at $ 0 $, then
	$ \gamma $ is horizontally bounded.
\end{lemma}
	
\begin{proof}
	Suppose that $ \gamma $ is a periodic hair horizontally unbounded and consider the
	connected components $ R_{k} $ of the inverse image of the real axis under 
	$ f_{\lambda} $ where $ R_{k} $ is the component that contains the point 
	$ k \pi, \quad k \in \mathbb{Z} $ (in the case 
	$ \lambda \in \mathbb{R}, \quad R_{k} $ will be the line 
	$ Re(z)=\frac{\pi}{2}+k \pi $). In the following we denote 
	$ f_{\lambda} \circ \gamma $ as $ \Gamma $.\\
	
	Since $ \Gamma $ converges uniformly to
	$ \infty $ and $ \gamma $ is horizontally unbounded, exists $ t_{0} \in \mathbb{R}^{+} $
	that satisfies the following properties:
	\begin{enumerate}
		\item $ \vert\Gamma(t)\vert > 2 \pi \qquad \forall t \geq t_{0} $;
		\item $ \gamma(t_{0})\in R_{k_{0}} $ for some
		$ k_{0} \in \mathbb{Z^{+}} $ (w.l.o.g.);
		\item $ \gamma(t) \in \mathbb{H} \qquad \forall t \geq t_{0} $ (w.l.o.g., by
		Lemma \ref{L:inH}), where $ \mathbb{H} $ is the upper half--plane.
	\end{enumerate}
	By Property 2, $ \Gamma(t_{0}) \in \mathbb{R} $. Furthermore, because $ \gamma $ is 
	horizontally unbounded, exists $ t_{1}>t_{0} $ such that 
	$ \gamma(t_{1}) \in R_{k_{0}+1} $;
	so, we can get $ t'_{0} $ and $ t'_{1} $ with 
	$ t_{0} \leq t'_{0} < t'_{1} \leq t_{1} $ and satisfying
	$ \gamma(t'_{0}) \in R_{k_{0}} \quad \gamma(t'_{1}) \in R_{k_{0}+1} $ and 
	$ \gamma(t) \in B_{k_{0}}\cap \mathbb{H} $ for all 
	$ t, \quad t'_{0} < t < t'_{1} $. Here,	$ B_{k_{0}} $ is the open strip 
	determined by $ R_{k_{0}} $ and $ R_{k_{0}+1} $.\\
	
	By Property 1, we can suppose w.l.o.g. that 
	$ \Gamma(t'_{0}) > 2 \pi $, so 
	$ \Gamma(t'_{1}) < -2 \pi $.
	Finally, consider the open set $ G $ bounded by 
	$ \Gamma_{1}(t) = \Gamma(t), \quad t'_{0} \leq t \leq t'_{1} $ and
	the real axis.\\

	Since $ f_{\lambda} $ is periodic, $ \Gamma_{2} = \Gamma_{1} + 2 \pi $ is
	a pre--periodic hair. Now, we have that there exist points of $ \Gamma_{2} $ inside
	and outside of $ G $, so $ \Gamma_{1} \cap \Gamma_{2} \neq \varnothing $. If
	$ w \in \Gamma_{1} \cap \Gamma_{2} $ then $ w $ and $ w-2 \pi $ are both in $ \Gamma $ 
	contradicting its injectivity as a periodic	hair, see Figure \ref{fig3}.\\
	
	\begin{figure}
	\begin{center}
		\includegraphics[width=0.8\linewidth]{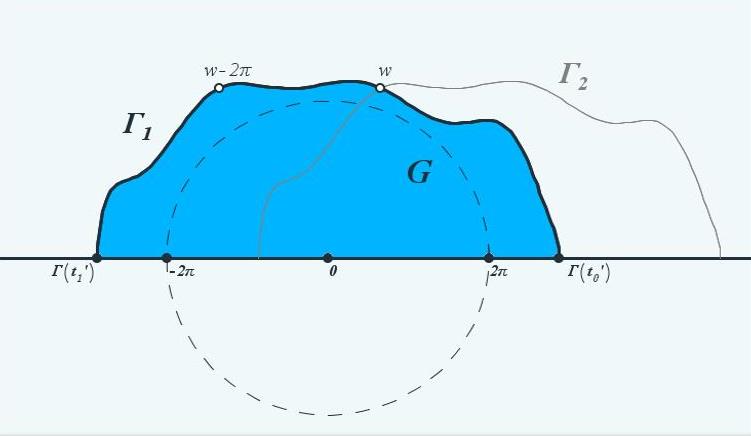}
	\end{center}
	\caption{ $ \Gamma_{1} $ must intersect its translate $ \Gamma_{2} $.}
	\label{fig3}
\end{figure}	
	
	Hence, $ \gamma $ must be horizontally bounded.
	
\end{proof}

Now we can prove

\begin{theorem}[Parabolic Case] \label{T:parabolic1}
	Let $ f_{ \lambda }(z)= \lambda \sin(z) $ with $ \lambda = e^{2 \pi i \frac{p}{q}} $, 
	$  p,q \in \mathbb{Z}, \quad q \neq 0 $ and $ (p,q)=1 $
	then all the components of $ F(f_{ \lambda }) $ are bounded. 

\end{theorem}

\begin{proof}
	Since $ f_{\lambda}^{n}(\lambda) \rightarrow 0 $ then the post--critical
	set of $ f_{\lambda} $ is bounded and considering that this function is
	criniferous and the origin is a parabolic point, there exists, by the result
	of Benini and Rempe--Gillen, a periodic hair $ \gamma $ landing at zero. So,
	$ \gamma \cup \lbrace 0 \rbrace \cup (- \gamma) $ together with all their
	translations by entire multiples of $ 2 \pi $ divide the complex plane into
	infinitely many strips $ B_{k}, \quad k \in \mathbb{Z} $, where each of the
	$ B_{k} $'s is horizontally bounded by Lemma \ref{L:realbounded}.\\
	
	Let $ G_{0} $ be a periodic component of the Fatou set associated to the
	parabolic point $ 0 $ such that $ G_{0} $ is unbounded. Because 
	$ G_{0} \subset B_{0} \cup B_{-1} $ (the only strips $ B_{i} $ such that 
	$ \gamma \subset \partial B_{i} $), then $ G_{0} $ must be vertically
	unbounded. Furthermore, the same can be said for all the parabolic periodic
	cycle and thus, all of them are contained in the set
	
	\[
		B=\lbrace z \in \mathbb{C} \vert -N < Re(z) < N \rbrace,
	\]
	for some $ N \in \mathbb{N} $ and we can suppose that $ N > 3 $.\\
	
	Let $ z_{0} = x_{0} + i y_{0} \in G_{0} $ such that $ \vert y_{0} \vert $
	is large enough to satisfy $ \sinh(\vert y_{0} \vert) > c_{1} N $ for some
	constant $ c_{1} > 10 $, for example.\\
	
	Thus $ \vert f(z_{0})\vert=\vert \sin (z_{0})\vert>\sinh\vert y_{0}\vert>c_{1} N $,
	it implies that $ f(z_{0}) $ is in the exterior of the circle with center
	at the origin and radius $ c_{1} N $, but inside $ B $. It implies that the
	absolute value of the imaginary part of $ z_{1} = f(z_{0}) = x_{1}+ i y_{1} $
	must be greater than $ \sqrt{(c_{1} N)^{2} - N^{2}} $, i.e.
	\[
	\vert y_{1} \vert > \sqrt{c_{1}^{2}-1} \cdot N.
	\]
	
	Now, we can estimate the modulus of $ z_{2} = f(z_{1}) = x_{2}+ i y_{2} $ as
	follows
	\begin{align*}
		\vert f(z_{1})\vert&=\vert \sin (z_{1})\vert>\sinh\vert y_{1}\vert>\sinh(\sqrt{c_{1}^{2}-1} \cdot N)\\
		&= \sinh\left( \frac{2\sqrt{c_{1}^{2}-1}}{2} N \right)\\
		&= 2 \sinh\left( \frac{\sqrt{c_{1}^{2}-1}}{2} N \right) \cosh\left( \frac{\sqrt{c_{1}^{2}-1}}{2} N \right).
	\end{align*}
	
	Therefore, $ z_{2} $ is in the exterior of the circle with center at the origin
	and radius
	\[
		R=2 \sinh\left( \frac{\sqrt{c_{1}^{2}-1}}{2} N \right) \cosh\left( \frac{\sqrt{c_{1}^{2}-1}}{2} N \right),
	\]
	but inside $ B $.\\
	
Thus, we can estimate the absolute value of the imaginary part of $z_{2}$ in the following way:
	\begin{align*}
		\vert y_{2} \vert &> \sqrt{2^{2} \sinh^{2}\left( \frac{\sqrt{c_{1}^{2}-1}}{2} N \right) \cosh^{2}\left( \frac{\sqrt{c_{1}^{2}-1}}{2} N \right)- N^{2}}\\
		&> \sqrt{2^{2} N^{2} \cosh^{2}\left( \frac{\sqrt{c_{1}^{2}-1}}{2} N \right)- N^{2}} & &\text{$(\sinh(x)>x)$}\\
		&=\left( \sqrt{2^{2} \cosh^{2}\left( \frac{\sqrt{c_{1}^{2}-1}}{2} N \right)- 1}\right)\cdot N.
	\end{align*}
	Observe that the last expression is again of the form $ \sqrt{c_{2}^{2}-1}\cdot N $. 
	Putting
	$ \alpha = \sqrt{c_{1}^{2}-1} \cdot N $ we can compare $ c_{1} $ with $ c_{2} $,
	this is,
	\begin{align*}
		c_{2} &= 2 \cosh\left( \frac{\alpha}{2} \right) > 2\left( \frac{e^{\frac{\alpha}{2}}}{2} \right)\\
		&> 1 + \frac{\alpha}{2} + \frac{\alpha^{2}}{8} > 1 + \frac{\alpha^{2}}{8}\\
		&=1+\frac{(c_{1}^{2}-1)N^{2}}{8}>c_{1}^{2} & &\text{$(N>3)$}.		 
	\end{align*}
	Therefore, $ c_{2}>c_{1}^{2} $ and so $ \vert y_{2}\vert > \sqrt{c_{1}^{4}-1} \cdot N $.
	The process can be continued indefinitely in such a way that if
	$ z_{n}=f_{\lambda}^{n}(z_{0})= x_{n} + i y_{n} $, then
	$ \vert y_{n} \vert > \sqrt{c_{1}^{2^{n}}-1} \cdot N $, so 
	$ \vert z _{n} \vert \rightarrow \infty $ when $ n \rightarrow \infty $. This 
	implies that $ z_{0} \in I(f_{\lambda})\subseteq J(f_{\lambda}) $ contradicting
	the choice of $ z_{0} $ in a Fatou component of $ f_{\lambda} $.
	So, all the periodic components of Fatou in this case are bounded and, because
	there is no asymptotic values for this family of functions and all the components
	are horizontally bounded then all of them are bounded. 
\end{proof}

\section{Repelling Case} \label{S:three}

We will use the same arguments developed in Section \ref{S:three} to prove

\begin{theorem}[Repelling Case] \label{T:repelling1}
	If $ \vert \lambda \vert>1 $ and $ \mathcal{P}(f_{ \lambda }) $ is bounded, then 
	all the components of $ F(f_{ \lambda }) $ are bounded. 

\end{theorem}

\begin{proof}
	The post--singular set of $ f_{\lambda} $ is bounded by hypothesis, thus the result of 
	Benini and Rempe--Gillen guarantees that $ 0 $ is the landing point of at least
	one periodic hair $ \gamma $ which is horizontally bounded by Lemma \ref{L:realbounded}. 
	So, $ \gamma \cup \lbrace 0 \rbrace \cup( -\gamma) $ together with all its 
	$ 2 \pi \mathbb{Z} $--translates divide the complex plane
	into strips $ B_{k} $. If a periodic component $ U $ of the Fatou set of $ f_{\lambda} $
	is unbounded, it must be completely contained in some of these strips and thus $ U $
	is horizontally bounded too. Therefore, the complete periodic cycle, being a finite
	union of horizontally bounded sets must be contained in the set
	\[
	\lbrace z \in \mathbb{C} \vert -N<Re(z)<N \rbrace \qquad \emph{for some} \quad N \in \mathbb{N}.
	\] 
	Using the same arguments that in the parabolic case and the fact that, in this
	case, $ \vert \lambda \sin(z) \vert > \vert \sin(z) \vert $ we have that all the
	components of any periodic cycle must be bounded. Finally, $ f_{\lambda} $
	does not have asymptotic values and the components are horizontally bounded, hence the 
	strictly pre--periodic components are bounded too.
\end{proof}

\section{The Extended Sine Family} \label{S:four}

In this section we will use the methods used in the proofs of the cases parabolic and
repelling for the sine family in order to show that the components of the Fatou set
of the functions in the extended family:
\[
	\fla(z)=\lambda \sin(z)+a, \qquad \lambda \in \mathbb{C}^{*}, \quad a \in \mathbb{C}
\]
are also bounded in the case $ \vert \lambda \vert > 1 $ with $ \mathcal{P}(\fla) $
bounded, supposing that $ \fla $ has not a Fatou component containing all the 
critical points.\\ 

Observe that Lemma \ref{L:inH} remains valid in the
case of the extended family because the image under $ \fla $ of any parallel line to the
real axis is an ellipse: a bounded set. Thus, a periodic hair (which converges uniformly
to infinity) of $ \fla $ landing at a periodic point can not intersect without bound
any of these parallel horizontal lines. This implies that $ \gamma(t) $
is eventually contained in any upper (w.l.o.g.) half--plane 
$ \mathbb{H}^{b}=\lbrace z \in \mathbb{C} \vert Im(z)>b>0 \rbrace $.\\

\emph{Afirmation.} The Lemma \ref{L:realbounded} is also true in the case of the extended sine family, where
$ \gamma $ is a periodic hair of $ \fla $ landing at a repelling or parabolic 
periodic point $ z_{0} $.

\begin{proof}
	Let $ L_{0}=\lbrace z \in \mathbb{C} \vert Im(z)=Im(\fla(z_{0})) \rbrace $. We only
	need to observe that, although the components of the inverse image of 
	$ L_{0} $ do not intersect $ L_{0} $ necessarily, $ \gamma $ is eventually
	contained in an upper half--plane $ \mathbb{H}^{*} $ whose boundary intersect
	all these components, see Figure \ref{fig4}. Hence exists $ t_{0} \in \mathbb{R}^{+} $
	that satisfies the following properties, where $ \Gamma = f_{\lambda}\circ\gamma$:
	\begin{enumerate}
		\item $ \vert \Gamma(t)- \fla (z_{0})\vert > 2 \pi \qquad \forall t \geq t_{0} $
		($ \Gamma $ converges uniformly to $ \infty $);
		\item $ \Gamma(t_{0}) \in L_{0} $;
		\item $ \gamma(t) \in \mathbb{H}^{*} \qquad \forall t \geq t_{0} $ (w.l.o.g., by
		Lemma \ref{L:inH} applied to the extended family).
	\end{enumerate}
	
\begin{figure}
	\begin{center}
		\includegraphics[width=0.8\linewidth]{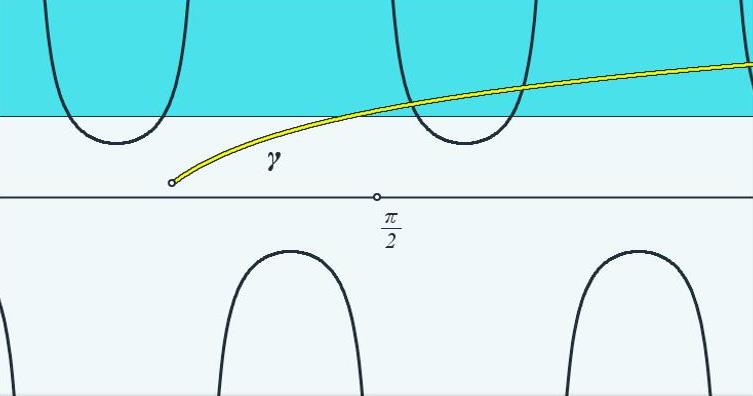}
	\end{center}
	\caption{ $ \gamma $ is eventually contained in the half--plane $ \mathbb{H}^{*} $.}
	\label{fig4}
\end{figure}	
	
	Hence, there exists $ t'_{0} $ and $ t'_{1} $ with $ t_{0}<t'_{0}<t'_{1} $ such
	that $ \lbrace \Gamma(t'_{0}), \Gamma(t'_{1}) \rbrace \subset L_{0} $ (by 
	Property 2) and	for all $ t \in (t'_{0},t'_{1}) $ we have that $ \Gamma(t) $ is
	in the intersection of the upper (w.l.o.g.) half--plane determined by $ L_{0} $
	and the exterior of the disk with center at $ \fla(z_{0}) $ and radius $ 2 \pi $
	(by Property 1), see Figure \ref{fig11}.\\
	
	As in the proof of Lemma \ref{L:realbounded}, at least one translate of $ \Gamma $
	must intersect $ \Gamma $ and the contradiction follows in the same way. 
\end{proof}

\begin{figure}
	\begin{center}
		\includegraphics[width=0.8\linewidth]{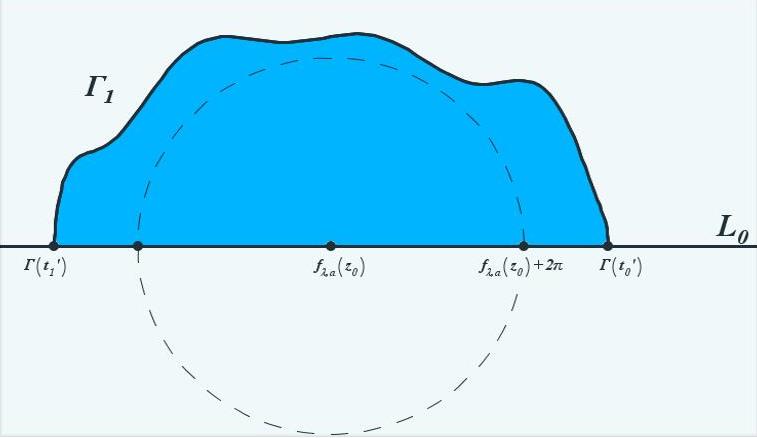}
	\end{center}
	\caption{ $ \Gamma $ and the line $ L_{0} $.}
	\label{fig11}
\end{figure}	

Furthermore, although the central symmetry in the origin has been lost, it is yet true
that $ \frac{\pi}{2}+ z \in F(\fla)$ if and only if 
$ \frac{\pi}{2}- z \in F(\fla)$.\\

The functions in the extended family are also contained in the Speiser class $\mathscr{S}$
and they are of finite order of growth,
so they are criniferous and the result of Benini and Rempe--Gillen
can be used in this context.

\begin{lemma}\label{L:allbounded}
	Let $ W_{1}, \ldots W_{m} $ a cycle of periodic components of the Fatou set of
	$ \fla $ with $ \vert \lambda \vert \geq 1 $ such that 
	$ \cup _{i=1} ^{m} W_{i} $ is horizontally bounded, then $ W_{i} $ is bounded
	for all $ i \in \lbrace 1, \ldots ,m \rbrace $.
\end{lemma}
 
\begin{proof}
	Suppose that the periodic cycle is contained in the strip 
	$ B=\lbrace z \in \mathbb{C} \vert -N < Re(z) < N \rbrace $ for some 
	$ N \in \mathbb{N} $. Observe that for the usual sine function 
	$ f_{\lambda} (z) = \lambda \sin (z) $ we can take $ y_{0} $ large enough as in the
	proof of Theorem (Parabolic Case) in such a way that
	we know that the imaginary part of 
	$ f_{\lambda} ^{n} (z_{0}) = x_{n} + i y_{n} $ goes to infinity when 
	$ n \rightarrow \infty $.\\
	
	Note that the imaginary part of $ z_{1} $ is greater than 
	$ \sinh \vert y_{0} \vert - \epsilon $ where $ \epsilon $ is as small as we
	want choosing $ y_{0} $ large enough; we simply use that 
	$ \epsilon < \vert a \vert $, see Figure \ref{fig5}. Additionally, we require 
	$ y_{0} $
	satisfying $ \sinh \vert y_{0} \vert >3 \vert a \vert $, 
	$ \cosh \vert y_{0} \vert > 4 $ and 
	$ \vert \lambda \vert e^{-y_{0}} < \vert a \vert $. Similar inequalities holds
	for all $ y_{n} $ because the sequence 
	$ ( \vert y_{n} \vert )_{n \in \mathbb{N}} $ is
	strictly increasing. Now we show that for a map in the extended sine family 
	$ \fla(z)= \lambda \sin(z)+a $ we can choose a point $ w_{0} $ in the vertically
	unbounded cycle of Fatou components of $ \fla $ such that the absolute values
	of the imaginary parts
	of its iterates under $ \fla $ is greater than the corresponding absolute values
	of the sequence $ y_{n} $.\\

\begin{figure}
	\begin{center}
		\includegraphics[width=0.8\linewidth]{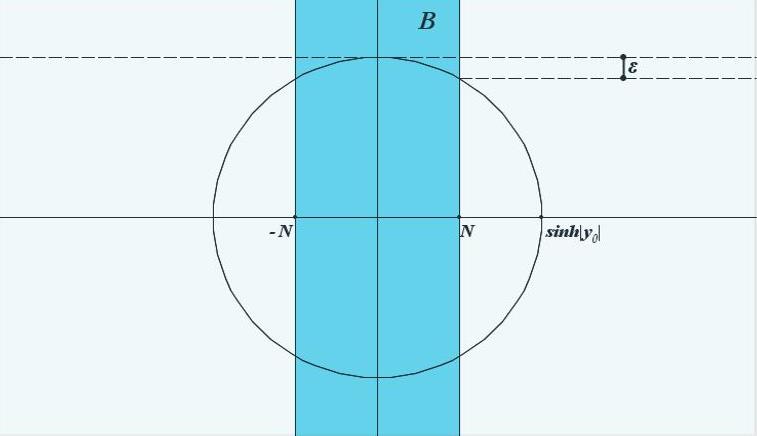}
	\end{center}
	\caption{ Meaning of $ \epsilon $.}
	\label{fig5}
\end{figure}	
	
	Choose $ w_{0} = u_{0} + i v_{0} $ such that 
	$ \vert v_{0} \vert > \vert y_{0} \vert + \vert a \vert $ and consider
	$ w_{n} = u_{n} + i v_{n} = \fla^{n} (w_{0}) $. We are going to prove by induction
	that $ \vert v_{n} \vert > \vert y_{n} \vert + \vert a \vert $.\\
	
	\begin{itemize}
		\item The case $ n = 0 $ is true by the choice of $ w_{0} $.
		\item Suppose that $ \vert v_{n} \vert > \vert y_{n} \vert + \vert a \vert $,
		then
		\[
			\vert y_{n+1} \vert < \vert z_{n+1} \vert < \vert \lambda \vert \cosh \vert  y_{n} \vert
		\]
		and
		\begin{align*}
			\vert v_{n+1} \vert &> \vert \lambda \vert \sinh \vert  v_{n} \vert - \epsilon -\vert a \vert \\
			&> \vert \lambda \vert \sinh (\vert  y_{n} \vert + \vert a \vert )- \epsilon -\vert a \vert \\
			&> \vert \lambda \vert (\sinh \vert  y_{n} \vert \cosh \vert a \vert + \cosh \vert y_{n} \vert \sinh \vert a \vert ) - \epsilon -\vert a \vert
		\end{align*}
		and thus, if $ \vert a \vert > 1 $
		\begin{align*}
			\vert v_{n+1} \vert-\vert y_{n+1} \vert &> \vert \lambda \vert \sinh \vert  y_{n} \vert \cosh \vert a \vert +  \vert \lambda \vert \cosh \vert y_{n} \vert \sinh \vert a \vert  - \epsilon -\vert a \vert - \vert \lambda \vert \cosh \vert y_{n} \vert \\
			&> \vert \lambda \vert \sinh \vert  y_{n} \vert \cosh \vert a \vert - \epsilon -\vert a \vert \\
			&> \sinh \vert  y_{n} \vert  - \epsilon -\vert a \vert \\
			&> 3 \vert a \vert - \vert a \vert - \vert a \vert = \vert a \vert. 
		\end{align*}
		In the case $ \vert a \vert < 1 $
		\begin{align*}
			\vert v_{n+1} \vert-\vert y_{n+1} \vert &> \vert \lambda \vert \sinh \vert  y_{n} \vert \cosh \vert a \vert +  \vert \lambda \vert \cosh \vert y_{n} \vert \sinh \vert a \vert  - \epsilon -\vert a \vert - \vert \lambda \vert \cosh \vert y_{n} \vert \\
			&> \vert \lambda \vert (\sinh \vert  y_{n} \vert -\cosh \vert y_{n} \vert)+ \vert \lambda \vert \cosh \vert y_{n} \vert \sinh \vert a \vert - \epsilon -\vert a \vert \\
			&> -\vert \lambda \vert e^{-\vert y_{0} \vert} + \vert \lambda \vert \cosh \vert  y_{n} \vert \vert a \vert  - \epsilon -\vert a \vert \\
			&> 4 \vert a \vert - 3 \vert a \vert  = \vert a \vert. \quad Done. 
		\end{align*}
	\end{itemize}
	So, if $ W_{i} $ is an unbounded (vertically) component of the Fatou set of
	$ \fla $, then $ w_{n} \in \cup _{i=1}^{m} W_{i} $ goes to infinity when 
	$ n \rightarrow \infty $. Hence $ w_{n} \in I( \fla ) \subseteq J( \fla ) $
	contradicting that $ w_{n} \in F( \fla ) $.\\
	
	Therefore, all the components of the Fatou set in the periodic cycle are bounded.
\end{proof}

In the case $ f_{\lambda}(z)= \lambda \sin (z) $ we have an evident fixed point
at $ 0 $ with multiplier $ \lambda $. For the maps $ \fla $ we have problems to find the exact localization of the fixed points and their respective multipliers. However,
because the map $ G(z)= \lambda \sin (z)- z $ is entire without asymptotic 
singularities we have that the equation 

\[
	\lambda \sin (z)- z = -a
\]
has an infinity of solutions (fixed points of $ \fla $) for all $ a \in \mathbb{C} $.\\

Furthermore, since $ \fla'(z)= \lambda \cos (z) $ and 
$ \fla (z)- a = \lambda \sin (z) $, we have that 

\[
	(\fla'(z))^{2}+(\fla(z)-a)^{2}=\lambda^{2}.
\]

So, if $ z_{0} $ is a fixed point of $ \fla $, then

\begin{equation} \label{E:multiplier}
	\fla'(z_{0})=\sqrt{\lambda^{2} - (z_{0} - a)^{2}}.
\end{equation}

\begin{definition}
	Let $ A $ be a connected subset of $ \mathbb{C} $. We say that $ A $ \textbf{has a 
	$ T^{+}- $end} if 
	$ A\cap\lbrace z \in \mathbb{C} \vert Im(z)>M \rbrace \neq \varnothing $ for all
	$ M \in \mathbb{N} $ but there exists $ M_{0} \in \mathbb{N} $ such that
	$ A\cap\lbrace z \in \mathbb{C} \vert Im(z)<-M_{0} \rbrace = \varnothing $. 
	$ A $ \textbf{has a	$ T^{-}- $end} if $ -A $ has a $ T^{+}- $end.  
\end{definition}

\begin{theorem}[Repelling Case in Extended Family]\label{T:extended}
	Let $ \fla(z)=\lambda \sin (z) + a $ with $ \vert \lambda \vert  > 1 $ a
	post--singularly bounded function, then all the components of the
	Fatou set of $ \fla $ are bounded if and only if for each of them exists a critical
	point that is not in the component.
\end{theorem}

\begin{proof}
	If there exists a component that contains all the critical points then, clearly, this
	component is unbounded, so the part ''only if'' is proved.\\
	
	On the other hand, using the result of Benini and Rempe, we have that each periodic
	repelling and parabolic point of 
	$ \fla $ is the landing point of a periodic hair. Suppose that two periodic
	hairs $ \gamma_{1} $ and $ \gamma_{2} $ land in a periodic point $ z_{0} $ of
	this type such that $ \gamma_{1} $ has a $ T^{+}- $end and
	$ \gamma_{2} $ has a $ T^{-}- $end. By Lemma \ref{L:realbounded} (using $ z_{0} $
	as landing point), $ \gamma_{1} $ and $ \gamma_{2} $ are horizontally bounded 
	and thus, $ \gamma_{1} \cup \lbrace z_{0} \rbrace \cup \gamma_{2}  $ together
	with all its $ 2 \pi \mathbb{Z} $--translates divide the plane in such a way
	that any component of the Fatou set of $ \fla $ must be horizontally bounded.
	Therefore, by Lemma \ref{L:allbounded} all the components of any periodic cycle
	of $ \fla $ are bounded and, because $ \fla $ has not asymptotic values, all the
	Fatou components of the map are bounded.\\
	
	The same occurs if there is a point $ z_{0} \in J(\fla) $ such that for every
	neighbourhood $ V $ of $ z_{0} $ there are $ z_{1} $ and $ z_{2} $ repelling
	periodic points in $ V $ with hairs $ \gamma_{1} $ and $ \gamma_{2} $ landing
	in them, respectively, and satisfying that $ \gamma_{1} $ has a $ T^{+}-end $
	and $ \gamma_{2} $ has a $ T^{-}-end $.\\
	
	The last possibility happens when there exists a Fatou component $ U $ of $ \fla $ that
	produces a disconnection between the set of all repelling or parabolic periodic
	points such that all their landing hairs have $ T^{+}-ends $ and the set of
	all repelling or parabolic points all of whose hairs have $ T^{-}-ends $.\\
	
	We have two cases:\\
	
	\begin{itemize}
		\item The real axis is contained in $ U $. In this case $ U $ contains all
		the critical points of the map contradicting that $ U $ must avoid at least one
		of them.
		\item There exists a point $ z_{0}=x_{0}+0i $ in the Julia set of $ \fla $ such
		that in some neighbourhood of $ z_{0} $ there are only repelling or parabolic
		periodic points whose landing hairs have (w.l.o.g.) $ T^{+}-ends $. We can
		choose $ z_{0} $ such that 
		$ -\frac{\pi}{2} < x_{0} <\frac{\pi}{2} $ due to $ 2 \pi $ periodicity and
		central symmetry of $ \fla $ in $ \frac{\pi}{2} $.\\
		
		Again, by central symmetry in $ \frac{\pi}{2} $, the point 
		$ z_{1} = \pi- x_{0} $ has a neighbourhood where every repelling or parabolic
		periodic points has all their landing hairs having $ T^{-}-ends $. Hence,
		by $ 2 \pi $--periodicity, $ U $ must slide above and below the real axis
		all the way.\\
		
		Moreover, $ U $ must contain both critical points $ \frac{\pi}{2} $ and
		$ -\frac{\pi}{2} $, otherwise, we can assume (w.l.o.g.) that
		$ \frac{\pi}{2} \notin U $ and we can consider the symmetrical Fatou domain $ U^{*} $
		of $ U $ with respect to $ \frac{\pi}{2} $. Now, if $ U=U^{*} $ then, since
		$ \frac{\pi}{2} \notin U $, $ U $ is infinitely connected, which is impossible.
		If $ U \neq U^{*} $ then there exists a periodic hair $ \Gamma $ between
		them, but in this case $ \Gamma $ would be horizontally unbounded 
		contradicting Lemma \ref{L:realbounded}. Hence 
		$ \lbrace \frac{\pi}{2},-\frac{\pi}{2} \rbrace \subset U $ and, by 
		periodicity, all the critical points are in $ U $. Therefore this case is also
		impossible.
		
	\end{itemize}
	
	In conclusion, all the components of the map $ \fla $  must be bounded.\\
		 
\end{proof}
		
\begin{theorem}[Disjoint Type]
	Suppose that $ U $ is an unbounded component of the Fatou set of $ \fla $, where
	$ \vert \lambda \vert > 1 $ and $ \mathcal{P}(\fla) $ is bounded. If $ U $ is forward 
	invariant then $ \fla $ is of disjoint type.
\end{theorem}
	
\begin{proof}
	By the proof of the previous Theorem we know that, since $ U $ is unbounded, it must
	contain all the critical points of $ \fla $. Also, $ U $ must contain the
	critical values because $ U $ is forward invariant. Hence, as $ U $ contains
	all the inverse images of any critical value, $ U $ is completely invariant.\\
	
	The classification of periodic components of the Fatou set gives three possible
	cases:
	
	\begin{itemize}
		\item $ U $ is a Siegel disc. This case is impossible because it is clear
		that $ \fla $ is not univalent in $ U $.
		\item $ U $ is a parabolic domain. In this case the multiplier of the fixed
		parabolic point $ z_{0} $ is equal to 1, so, by the equation (\ref{E:multiplier}):

		\begin{align*}
			1 &= \lambda^{2} - (z_{0}-a)^{2}\\
			z_{0} &= \sqrt{\lambda^{2}-1}+a
		\end{align*}
		
		Because $ \lambda \sin (z_{0})+a=z_{0}= \sqrt{\lambda^{2}-1}+a $, we have that
		$ \lambda \sin (z_{0}) = \sqrt{\lambda^{2}-1} $, furthermore, 
		$ \fla'(z_{0}) =1 $ implies that $ \lambda \cos (z_{0}) = 1 $. So,
		
		\[
			\frac{i}{\lambda}\sqrt{\lambda^{2}-1}+\frac{1}{\lambda}= i \sin(z_{0})+ \cos(z_{0})= e^{i z_{0}}
		\]
		
		It means that $ e^{i z_{0}} $ does not depend of $ a $, and thus 
		$ a=2 \pi k_{0} $ for some $ k_{0} \in \mathbb{Z} $. We remark that, in this
		case, $ 2 \pi k_{0} $ is a repelling fixed point of the map 
		$ f_{\lambda,2 \pi k_{0}} $, which implies that there exists a periodic hair
		$ \gamma $ landing at $ 2 \pi k_{0} $.\\
		
		Now, the map $ f_{\lambda,2 \pi k_{0}} $ has a central symmetry with respect
		to $ 2 \pi k_{0} $ as the following calculation shows:\\
		
		The points $ 2 \pi k_{0} + z $ and $ 2 \pi k_{0} - z $ are symmetrical with
		respect to $ 2 \pi k_{0} $. Applying the map to each of them we have
		
		\begin{align*}
			f_{\lambda,2 \pi k_{0}}( 2 \pi k_{0} + z ) &= \lambda \sin (2 \pi k_{0} + z)+ 2 \pi k_{0}\\
			&= \lambda \sin (z) + 2 \pi k_{0}
		\end{align*}
		
		and
		
		\begin{align*}
			f_{\lambda,2 \pi k_{0}}( 2 \pi k_{0} - z ) &= \lambda \sin (2 \pi k_{0} - z)+ 2 \pi k_{0}\\
			&= \lambda \sin (-z) + 2 \pi k_{0}\\
			&= -\lambda \sin (z) +  2 \pi k_{0}
		\end{align*}
		
		Hence, the images under $ f_{\lambda,2 \pi k_{0}} $ of two 
		$ 2 \pi k_{0} $--central symmetric points are $ 2 \pi k_{0} $--central symmetric
		points and thus, if $ \gamma_{1} $ is the curve
		
		\[
			\gamma_{1}=\lbrace 2 \pi k_{0}-z \in \mathbb{C} \vert 2 \pi k_{0}+z \in \gamma \rbrace
		\]
		
		then $ \gamma \cup \lbrace 2 \pi k_{0} \rbrace \cup \gamma_{1} \subset J(f_{\lambda,2 \pi k_{0}}) $ divide the plane contradicting that $ U $ is horizontally unbounded.
		\item $ U $ is an attractive domain. Since $ U $ is completely invariant,
		$ U=F(\fla) $ contains all the critical points and critical values of $ \fla $ 
		(Theorem 6 in \cite{EL1992}) which implies that $ z_{0} $ is the unique attracting 				fixed point. Therefore, $ \fla $ is of disjoint type.
	\end{itemize}	   
	
\end{proof}

We conjecture that, if $ \fla $ is post--critically bounded, $ \vert \lambda \vert >1$
and $ U $ is an unbounded component of the Fatou set of $ \fla $ then $ \fla $ is of
disjoint type, this is $ U $ is completely invariant.\\

Finally, we show that it is possible to find examples of maps $ \fla $ of disjoint type
with $ \vert \lambda \vert > 1 $. Consider the map $ \fla $ with $ \lambda=-1.003 $
and $ a=\frac{\pi}{2}-1.003 $. This map satisfies that 
$ \fla([-\frac{\pi}{2},\frac{\pi}{2}])\subset [-\frac{\pi}{2},\frac{\pi}{2}] $ and
$ \fla(-\frac{\pi}{2})=\frac{\pi}{2} $, so all the critical points converge to the same attracting fixed point $ z_{0}\approx 0.28541 $ 
and the map is of disjoint type, see Figure \ref{fig6}.\\

\begin{figure}
	\begin{center}
		\includegraphics[width=0.8\linewidth]{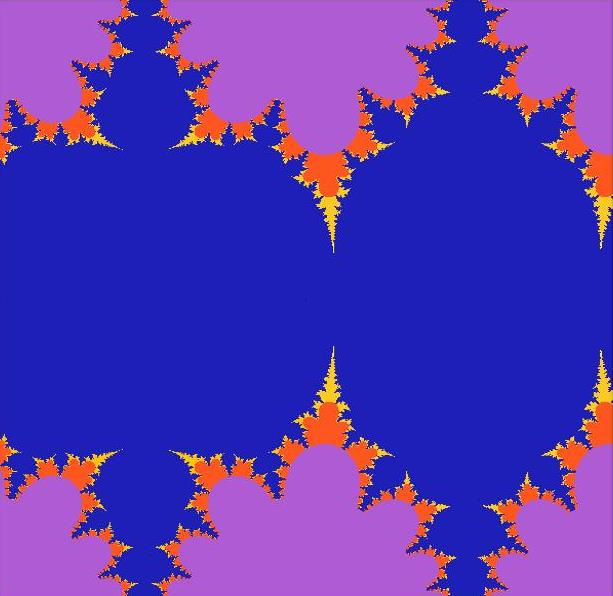}
	\end{center}
	\caption{ Dynamical plane for the map $ f_{-1.003,\frac{\pi}{2}-1.003} $.}
	\label{fig6}
\end{figure}

In this example, the parameters are real numbers, this allows to show easily that the 
map is of disjoint type, however, being a structurally stable map, there are a lot of
examples of disjoint type maps $ \fla $ with $ \vert \lambda \vert > 1 $ and 
complex parameters, see Figures \ref{fig7}, \ref{fig8} and \ref{fig9}.\\

\begin{figure}
	\begin{center}
		\includegraphics[width=0.5\linewidth]{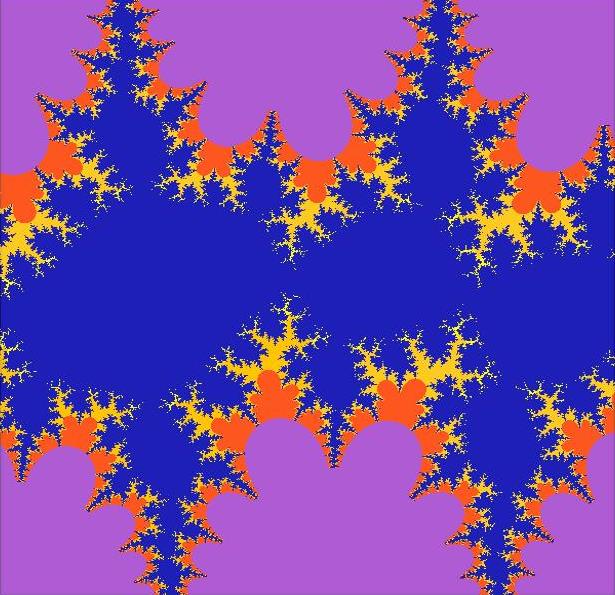}
	\end{center}
	\caption{ Dynamical plane for the map $ f_{1.2 e^{\frac{83}{90} \pi i},1.4+0.25i} $.}
	\label{fig7}
\end{figure} 

\begin{figure}
	\begin{center}
		\includegraphics[width=0.5\linewidth]{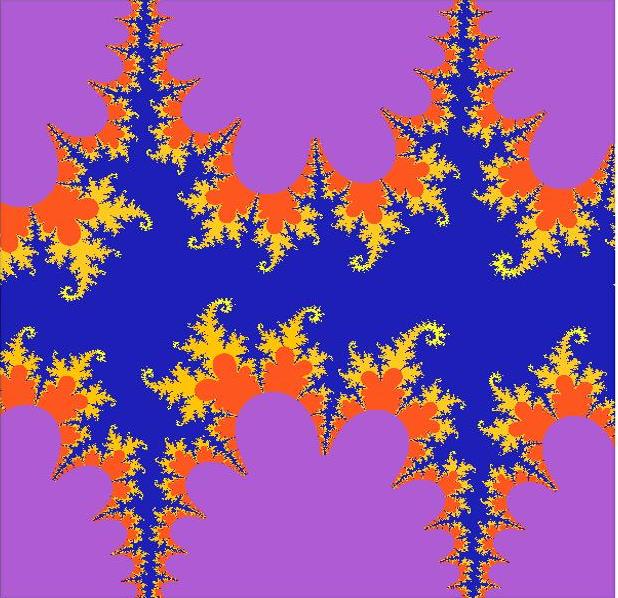}
	\end{center}
	\caption{ Dynamical plane for the map $ f_{1.505 e^{\frac{89}{90} \pi i},2+0.12i} $.}
	\label{fig8}
\end{figure}

\begin{figure}
	\begin{center}
		\includegraphics[width=0.5\linewidth]{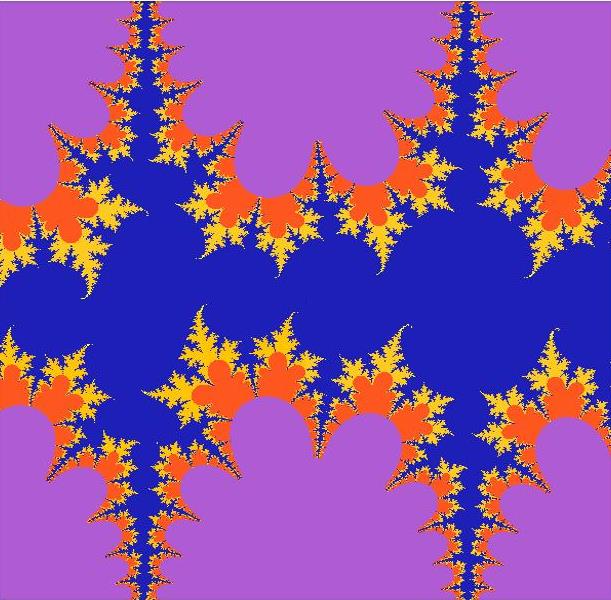}
	\end{center}
	\caption{ Dynamical plane for the map $ f_{-1.5,2+0.08i} $.}
	\label{fig9}
\end{figure}


\begin{thebibliography}{99}
	\bibitem{B1975}
	Baker I. N. [1975] ''The domains of normality of an entire function'', 
	\emph{Ann. Acad. Sci. Fenn. Ser. A I Math.} \textbf{1}, 277-283.
	\bibitem{B1991}
	Beardon A. F. [1991] ''Iteration of Rational Functions'', \emph{Graduate Texts
	in Mathematics. Vol. 132. Springer--Verlag, NY}.
	\bibitem{BR2017}
	Benini A. M., Rempe--Gillen L. [2017] ''A landing theorem for entire functions with
	bounded post--singular sets'' \emph{Preprint, arXiv:1711.10780}
	\bibitem{BFR2015}
	Bergweiler W., Fagella N., Rempe--Gillen L. [2015] ''Hyperbolic entire functions
	with bounded Fatou components'' \emph{Comm. Math. Helv.} \textbf{90}, 799--829.
	\bibitem{DT1986}
	Devaney R., Tangerman F. [1986] ''Dynamics of entire functions near the essential
	singularity'', \emph{Ergod. Th. Dyn. Syst.} \textbf{6}, 489--503.
	\bibitem{DF2008}
	Dom\'{i}nguez P., Fagella N. [2008] ''Residual julia sets of rational and
	transcendental functions'', \emph{Transcendental Dynamics and Complex Analysis},
	138--164.
	\bibitem{DS2002}
	Dom\'{i}nguez P., Sienra G. [2002] ''A study of the dynamics of 
	$ \lambda \sin(z) $'', \emph{International Journal of Bifurcation and Chaos}
	\textbf{12}, 2869--2883.
	\bibitem{E1989}
	Eremenko A. E. [1989] ''On the iteration of entire functions'' \emph{Dyn. Syst. 
	Ergod. Th. Banach Center Publications} \textbf{23}, 489--503.
	\bibitem{EL1992}
	Eremenko A. E., Lyubich Y. [1992] ''Dynamical properties of some classes of 
	entire functions'', \emph{Ann. Inst. Fourier, Grenoble} \textbf{42}, 989--1020.
	\bibitem{HY1998}
	Hua X., Yang Ch. [1998] ''Dynamics of transcendental functions'' \emph{Asian 
	Mathematics Series, Gordon and Breach Science Publishers}.
	\bibitem{M2010}
	Mihaljevi\'{c}--Brandt H. [2010] ''A landing theorem for dynamic rays of 
	geometrically finite entire functions'' \emph{J. of the London Mathematical
	Society} \textbf{81}, 696--714.
	\bibitem{M1999}
	Milnor J. [1999] ''Dynamics in One Complex Variable. Introductory Lectures'' 
	\emph{Vieweg and Sohn Press}.
	\bibitem{R2008}
	Rempe L. [2008] ''Siegel disks and periodic rays of entire functions'',
	\emph{J. Reine Angew. Math.} \textbf{624}, 81--102. 
	\bibitem{RRRS2011}
	Rottenfusser G., R\"{u}ckert J., Rempe L., Schleicher D. [2011] ''Dynamic rays
	of bounded--type entire functions'', \emph{Ann. of Math. (2)} \textbf{173}, 
	no. 1, 77--125.
	\bibitem{S2018}
	Sixsmith D. [2018] ''Dynamics in the Eremenko--Lyubich Class'', \emph{Conformal
	Geometry and Dynamics} \textbf{22}, 185--224.
	\bibitem{Z2005}	
	Zhang G. [2005] ''On the dynamics of $ e^{2 \pi i \theta} \sin(z) $'', 
	\emph{Illinois J. Math.} \textbf{49}, 1171--1179. 
	
\end{thebibliography}
\end{document}